\documentclass[12pt]{amsart}
\usepackage{amsfonts,amsmath,amssymb,amscd,amsthm,enumerate,tikz}
\usepackage{amsrefs}
\usepackage{geometry}
\usepackage{adjustbox}
\usepackage[all]{xy}
\usepackage{stmaryrd}
\usetikzlibrary{arrows.meta,decorations.markings}

\newtheorem{theorem}{Theorem}
\newtheorem{prop}[theorem]{Proposition}
\newtheorem{lem}[theorem]{Lemma}
\newtheorem{cor}[theorem]{Corollary}
\theoremstyle{definition}

\newtheorem{rem}[theorem]{Remark}
\newtheorem{mydef}[theorem]{Definition}
\newtheorem{example}[theorem]{Example}

\renewcommand{\epsilon}{\varepsilon}

\def\R{\mathbb{R}}

\def\x{\mathbf{x}}

\def\<{\langle}
\def\>{\rangle}

\def\LL{{\mathcal L}}

\newcommand{\BigWedge}{\mathord{\adjustbox{valign=B,totalheight=.7\baselineskip}{$\bigwedge$}}}

\DeclareMathOperator{\spn}{span}

\usepackage{multicol} 

\usepackage{tikz}
\usepackage{caption}
\usepackage{subcaption}

\begin{document}
\title[On the cohomology of Lie algebras associated with graphs]{On the cohomology of Lie algebras associated with graphs}
\author[M.\ Aldi, A.\ Butler, J.\ Gardiner, D.\ Grandini, M.\ Lichtenwalner, K.\ Pan]{Marco Aldi, Andrew Butler, Jordan Gardiner, Daniele Grandini, Monica Lichtenwalner, Kevin Pan}

\begin{abstract}
We describe a canonical decomposition of the cohomology of the Dani-Mainkar 2-step nilpotent Lie algebras associated with graphs. As applications, we obtain explicit formulas for the third cohomology of any Dani-Mainkar Lie algebra and for the cohomology in all degrees of Lie algebras associated with arbitrary star graphs. We also describe a procedure to reduce the calculation of the cohomology of solvable Lie algebras associated with graphs through the Grantcharov-Grantcharov-Iliev construction to the cohomology of Dani-Mainkar Lie algebras. 
\end{abstract}
\maketitle


\section{Introduction}
Originally motivated by the study of Anosov diffeomorphisms on nilmanifolds, the Dani-Mainkar construction \cite{DaniMainkar05} associates with each simple graph $G$ a 2-step nilpotent Lie algebra $\mathcal L(G)$. While a theorem of Mainkar \cite{Mainkar15} shows there is no loss of information in going from a graph $G$ to its Dani-Mainkar Lie algebra $\mathcal L(G)$, this information is repackaged in a nontrivial way. With this dictionary at hand, it is rather natural to consider the Cartan-Chevalley-Eilenberg complex $\mathcal C^\bullet(G)$ of the Lie algebra $\mathcal L(G)$. Thanks to a theorem of Nomizu \cite{Nomizu54}, the cohomology $H^\bullet(\mathcal L(G))$ of $\mathcal C^\bullet(G)$ is isomorphic to the de Rham cohomology of nilmanifolds of the form $\exp(\mathcal L(G))/\Lambda$ for a suitable discrete co-compact group $\Lambda$. The first and second cohomology of $\mathcal L(G)$ have been calculated in \cite{PouseeleTirao09} for any simple graph $G$. 

Our starting point is the observation that the Cartan-Chevalley-Eilenberg complex of $\mathcal L(G)$ decomposes canonically into direct summands labeled by induced subgraphs of $G$. In degree one and two, this decomposition yields a different proof of the results of \cite{PouseeleTirao09}. In degree three, we obtain a novel and explicit formula for the third cohomology as a weighted count of induced subgraphs from a list of 23 possible isomorphism types. While in principle our analysis produces formulas in all degrees, the list of possible isomorphism types grows too rapidly for our method to be of practical use in calculating high degree cohomology of general Dani-Mainkar Lie algebras. Nevertheless, our decomposition can be used to effectively calculate the cohomology of $\mathcal L(G)$ in all degrees for special families of graphs that are closed under taking induced subgraphs. As an illustration, we calculate the full cohomology $H^\bullet(\mathcal L(G))$ when $G$ is a star graph of arbitrary size. While this result is known \cite{ArmstrongCairnsJessup97}, our methods are different and provide new information regarding the contribution of the induced subgraphs. Moreover we expect our techniques to be applicable to other classes of graphs with relatively simple induced subgraphs.

Following work by Grantcharov, Grantcharov, and Iliev \cite{GGI} we also consider solvable extensions $\mathcal L(G,\Sigma)$ of the abelian Lie algebra generated by a collection $\Sigma$ of cliques of $G$ by $\mathcal L(G)$. Our main result in this direction completely reduces the calculation of the cohomology of Grantcharov-Grantcharov-Iliev Lie algebras to the Dani-Mainkar case. More precisely, we show that given $G$ and $\Sigma$ as above there exists another (explicitly constructed) finite simple graph $\widetilde{G}$ such that the cohomology of $\mathcal L(G,\Sigma)$ and the cohomology of $\mathcal L(\widetilde{G}+|\Sigma|K_1)$ are isomorphic as vector spaces. In particular, combining this with our results regarding the cohomology of the Dani-Mainkar Lie algebras, we obtain explicit formulas for the dimension of low-degree cohomology groups of a large class of solvable Lie algebras and, thanks to a theorem of Hattori \cite{Hattori60}, of the corresponding solvmanifolds.

\subsection*{Acknowledgements} The work on this paper was supported by the National Science Foundation grant number DMS1950015 and by VCU Quest Award ``Quantum Fields and Knots: An Integrative Approach.'' M.A.\ is grateful to Sam Bevins for useful conversations.

\section{Background}

\subsection{The Dani-Mainkar Construction}

\begin{mydef}
Let $G$ be a finite simple graph with set of vertices $V(G)=\{1,\ldots,n\}$ and set of edges $E(G)$. Let $V$ be the real vector space with basis $x_1,\ldots,x_n$ and let $W$ be the subspace of $\BigWedge^2 V$ generated by all quadratic monomials of the form $x_i\wedge x_j$ whenever $\{i,j\}\notin E(G)$. Define the real vector space $\mathcal L(G)$ as the direct sum of $V$ with the quotient $(\BigWedge^2 V)/W$. Then $\mathcal L(G)$ has a canonical structure of 2-step nilpotent Lie algebra such that $[x,x']=x\wedge x' \mod W$ for all $x,x'\in V$, and $[\BigWedge^2 V,\mathcal L(G)]=0$. We refer to $\mathcal L(G)$ as the {\it Dani-Mainkar Lie algebra of $G$}.
\end{mydef}

\begin{rem}
In practice, it is convenient to think of $\mathcal L(G)$ as having basis $x_1,\ldots,x_n$ and $x_{i,j}$ whenever $\{i,j\}\in E(G)$ and $i<j$. Then the nontrivial brackets between generators are $[x_i,x_j]=x_{i,j}$ if $\{i,j\}\in E(G)$ and $i<j$, as well as $[x_i,x_j]=-x_{i,j}$ if $\{i,j\}\in E(G)$ and $i>j$. 
\end{rem}

\begin{example} $\mathcal L(K_2)$ is spanned by $\{x_1, x_2, x_{1,2}\}$ with bracket $[x_1, x_2] = x_{1,2}$ and all other brackets not defined by bilinearity or skew symmetry equaling zero. In other words, $\mathcal L(K_2)$ is the three-dimensional Heisenberg Lie algebra.
\end{example}

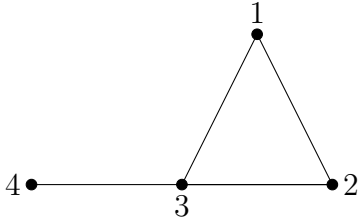
\begin{figure}\begin{tikzpicture}
    \filldraw
    (-3,-2) circle (2pt) node[left] {$4$}
    (-1,-2) circle (2pt) node[below] {$3$}
    (1,-2) circle (2pt) node[right] {$2$}
    (0,0) circle (2pt) node[above] {$1$};
    
    \draw (1,-2) -- (-1,-2) -- (0,0) -- (1,-2) -- (-3,-2);
\end{tikzpicture}
\caption{The graph $G$ as defined in Example \ref{ex:4}.}\label{fig:0}
\end{figure}

\begin{example}\label{ex:4}
If $G$ is the graph in Figure \ref{fig:0}, then $\mathcal{L}(G)$ is spanned by $x_1$, $x_2$, $x_3$, $x_4$, $x_{1,2}$, $x_{1,3}$, $x_{2,3}$, $x_{3,4}$ with brackets $[x_1, x_2] = x_{1,2}$, $[x_1, x_3] = x_{1,3}$, $[x_2, x_3] = x_{2,3}$, $[x_3, x_4] = x_{3,4}$. All other brackets not defined by bilinearity or skew symmetry are set to zero. 
\end{example}

\subsection{The Grantcharov-Grantcharov-Iliev Construction}

\begin{mydef}\label{def:GGI}
Let $G$ be a finite simple graph with vertices $V(G)=\{1,\ldots,n\}$ and let $\Sigma=\{\sigma_1,\ldots,\sigma_s\}$ be an arbitrary collection of (possibly repeated, not necessarily of the same size) cliques of $G$. Consider the real vector space $\mathcal L(G,\Sigma)$ with generators $x_1,\ldots,x_n,y_1,\ldots,y_s$, as well as $x_{i,j}$ whenever $1\le i<j\le n$ and $\{i,j\}\in E(G)$. We endow $\mathcal L(G,\Sigma)$ with a bilinear skew-symmetric bracket $[-\,,-]$ such that 
\begin{enumerate}[1)]
\item $[x_i,x_j]=x_{i,j}$ whenever $i<j$ and $\{i,j\}\in E(G)$,
\item $[x_i,y_k]=|\{i\}\cap V(\sigma_k)|x_i$ for all $i\in \{1,\ldots,n\}$, $k\in\{1,\ldots,s\}$
\item $[x_{i,j},y_k]=|\{i,j\}\cap V(\sigma_k)|x_{i,j}$ for all $1\le i<j\le n$ such that $\{i,j\}\in E(G)$ and $k\in \{1,\ldots,s\}$. 
\end{enumerate}
\end{mydef}

\begin{example}
The particular case of solvable Lie algebras of the form $\mathcal L(G,\Sigma)$ where $G$ is a finite simple graph and, for some $k>1$, $\Sigma$ is the collection of all $k$-cliques of $G$ is the one originally studied in \cite{GGI}.    
\end{example}

\begin{rem}
    If $\Sigma$ is the empty set, then $\mathcal L(G,\Sigma)$ coincides with the Dani-Mainkar Lie algebra $\mathcal L(G)$.
\end{rem}

\begin{example}\label{ex:9}
    Let $G$ be the graph of Figure \ref{fig:0} 
    and let $\Sigma=\{\sigma_1,\sigma_2,\sigma_3\}$ with $\sigma_1=\sigma_2=\{1,2\}$ and $\sigma_3=\{1,2,3\}$.
    Then 
    \begin{equation}
        \mathcal L(G,\Sigma)=\spn_\R\{x_1,x_2,x_3,x_4,x_{1,2},x_{1,3},x_{2,3},x_{3,4},y_1,y_2,y_3\}\,.
        \end{equation}
    In addition to the brackets of $\mathcal L(G)$ listed in Example \ref{ex:4}, we have $[x_1,y_1]=[x_1,y_2]=[x_1,y_3]=x_1$, $ [x_2,y_1]=[x_2,y_2]=[x_2,y_3]=x_2$, $[x_3,y_3]=x_3$, $[x_{1,2},y_1]=[x_{1,2},y_2]=[x_{1,2},y_3]=2x_{1,2}$, $2[x_{2,3},y_1]=2[x_{2,3},y_2]=[x_{2,3},y_3]=2x_{2,3}$, $2[x_{1,3},y_1]=2[x_{1,3},y_2]=[x_{1,3},y_3]=2x_{2,3}$, and $[x_{3,4},y_3]=x_{3,4}$.
\end{example}

\begin{rem}\label{rem:10}
Let $G$ and $G'$ be finite simple graphs and let $\Sigma$ and $\Sigma'$ be collections of cliques of $G$ and $G'$, respectively. Then $\mathcal L(G+G',\Sigma\cup \Sigma')=\mathcal L(G,\Sigma)\oplus \mathcal L(G',\Sigma')$. 
\end{rem}

\begin{prop}\label{prop:10}
Let $G$ be a finite simple graph and let $\Sigma$ be a collection of cliques of $G$. Then $\mathcal L(G,\Sigma)$ is a completely solvable Lie algebra.
\end{prop}

\begin{proof}
It suffices to verify the Jacobi identity on generators. There are two cases in which the three terms are not identically zero. 

\textit{Case 1.} Consider $\{i,j\}\in E(G)$ and $\sigma_k\in \Sigma$ such that $[x_i,y_k]=\alpha_i x_i$, $[x_j,y_k]=\alpha_j x_j$, and $[x_{i,j},y_k]=\alpha_{i,j} x_{i,j}$ for some $\alpha_i,\alpha_j,\alpha_{i,j}\in \{0,1,2\}$. By inspection, $\alpha_i+\alpha_j=\alpha_{i,j}$ vanishes if $\{i,j\}\cap V(\sigma_k)=\{\emptyset, \{i\}, \{j\}, \{i,j\}\}$ yielding

\begin{align*}
        [x_i,[x_j,y_k]]+[x_j,[y_k,x_i]]+[y_k,[x_i,x_j]] & =\alpha_j[x_i,x_j]-\alpha_i[x_j,x_i]+\alpha_{i,j}[y_k,x_{ij}]\\
        & =(\alpha_i+\alpha_j-\alpha_{i,j})x_{i,j}\\
        & = 0\,.
\end{align*}

\textit{Case 2.} Let $\sigma_i,\sigma_j\in \Sigma$ and let $x\in \mathcal L(G)$ be a variable corresponding to either a vertex or an edge of $G$ so that $[x,y_i]=\alpha_i x$ and $[x,y_j]=\alpha_j x$ for some $\alpha_i,\alpha_j\in \{0,1,2\}$. Then 
\begin{equation*}
 [x,[y_i,y_j]]+[y_i,[y_j,x]]+[y_j,[x,y_i]] = 0 + \alpha_i\alpha_j x -\alpha_j\alpha_i x =0\,.
\end{equation*}
We conclude that $\LL(G,\Sigma)$ is a Lie algebra. Next, we show that $\LL(G,\Sigma)$ is completely solvable i.e.\  admits a finite chain of ideals $0= I_0\subseteq I_1\subset \cdots \subset I_n  = \mathcal L(G,\Sigma)$ so that $I_{i-1}$ is of codimension 1 in $I_i$ for each $i\in {1,\ldots,n}$. Let $\Sigma=\{\sigma_1,\sigma_2,\ldots,\sigma_s\}$  and let $\Sigma'=\{\sigma_2,\ldots,\sigma_s\}$. Then $\mathcal L(G,\Sigma')$ is an ideal of codimension one in $\mathcal L(G,\Sigma)$. Iterating this process, we obtain a finite ascending chain of ideals $\mathcal L(G)\subset J_1\subset J_2 \cdots \subset \mathcal L(G,\Sigma)$ each of codimension 1 into the next. Since $\mathcal L(G)=\mathcal L(G,\emptyset)$ is nilpotent, it is also completely solvable. Hence $\mathcal L(G,\Sigma)$ is completely solvable.
\end{proof}

\section{Canonical Decomposition of the cohomology of $\mathcal L(G)$}\label{sec:3}

\begin{mydef}
Let $G$ be a finite simple graph. We denote by $\mathcal C^\bullet(G)$ the {\it Cartan-Chevalley-Eilenberg complex}  of the Dani-Mainkar Lie algebra $\mathcal L(G)$. As a vector space, $\mathcal C^\bullet(G)$ is the full exterior algebra of the dual of $\mathcal L(G)$. Equivalently, $\mathcal C^\bullet(G)$ can be thought of as the space of skew-commuting real polynomials in the variables $\{x_i^*\}_{i\in V(G)}$ and $\{x_e^*\}_{e\in E(G)}$. The differential $Q$ acting on $\mathcal C^\bullet(G)$ is the odd derivation such that $Q(x_{i,j}^*)=x_i^*x_j^*$ for all $\{i,j\}\in E(G)$ such that $i<j$ and $Q(x_i^*)=0$ for all $i\in V(G)$. We denote by $H^*(\mathcal L(G))$ the cohomology of $\mathcal C^\bullet(G)$ and write $b_i(G)$ for the dimension of $H^i(\mathcal L(G))$. 
\end{mydef}

\begin{mydef}
Let $G$ be a finite simple graph. Given any subset $S\subseteq V(G)$, we denote the corresponding {\it induced subgraph} by $G[S]$ so that $V(G[S])=S$ and $E(G[S])$ consists of all $e\in E(G)$ such that both vertices of $e$ are in $S$. We define the  subcomplex $\mathcal C^\bullet(G,S)$ of $\mathcal C^\bullet(G)$ whose underlying vector space is spanned by the set of all monomials of the form
\begin{equation}\label{eq:7}
x_{i_1}^*x_{i_2}^*\cdots x_{i_p}^*x_{r_1,s_1}^*,\cdots,x_{r_q,s_q}^*
\end{equation}
for some $i_1,\ldots,i_p\in V(G)$ and $\{r_1,s_1\},\ldots,\{r_q,s_q\}\in E(G)$ such that (allowing for possible repetitions)
\begin{equation}
S=\{i_1,\ldots,i_p,r_1,s_1,\dots,r_q,s_q\}\,. 
\end{equation}
By convention we set $\mathcal C^\bullet(G,\emptyset)$ to be the trivial complex whose graded components are all zero-dimensional.
\end{mydef}

\begin{rem}
The fact that, as the terminology suggests, $\mathcal C^\bullet(G,S)$ is indeed closed under the action of $Q$ follows immediately from the definition of the latter as a derivation.
\end{rem}

\begin{example}
Let $G$ be the graph in Figure \ref{fig:0} and let $S=\{1,3,4\}$. Then $\mathcal C^1(G,S)=0$, while $\mathcal C^2(G,S)$ is spanned by $x_1^*x_{3,4}^*$, $x_4^*x_{1,3}^*$, and $x_{1,3}^*x_{3,4}^*$.
\end{example}

\begin{mydef}
Let $G$ be a finite simple graph. We denote the cohomology of the complex $\mathcal C^\bullet(G,V(G))$ by $\mathcal H^\bullet (G)$ and call it the {\it essential cohomology} of $G$. We write $\beta_i(G)$ for the dimension of $\mathcal H^i(G)$.
\end{mydef}

\begin{rem}\label{rem:19}
If $S$ is a subset of the vertices of the disjoint union $G_1+G_2$ of the finite simple graphs $G_1$ and $G_2$, then we have a canonical isomorphism of cochain complexes 
\begin{equation}
\mathcal C^\bullet(G_1+G_2,S)\cong \mathcal C^\bullet (G_1,S\cap V(G_1))\otimes \mathcal C^\bullet (G_2,S\cap V(G_2))\,.
\end{equation}
In particular, we obtain
\begin{equation}\label{eq:10}
\beta_i(G_1+G_2)=\sum_{j+k=i}\beta_j(G_1)\beta_k(G_2)
\end{equation}
For instance, since $\beta_i(K_1)=\delta_{i,1}$, it follows from \eqref{eq:10} that $\beta_i(nK_1)=\delta_{i,n}$ for all $i,n\ge 1$.
\end{rem}

\begin{theorem}\label{thm:20}
As a graded vector space, the cohomology of $\mathcal L(G)$ is canonically isomorphic to the direct sum of the essential cohomologies of the induced subgraphs of $G$. In particular,
\begin{equation}\label{eq:9}
b_i(G)=\sum_{S\subseteq V(G)} \beta_i(G[S])\,.
\end{equation}
\end{theorem}

\begin{proof}
Since each $\mathcal C^\bullet(G,S)$ is a subcomplex and each monomial expressed as in \eqref{eq:7} is contained in precisely one of the $C^\bullet(G,S)$, we have the canonical decomposition
\begin{equation}
\mathcal C^\bullet(G) \cong \bigoplus_{S\subseteq V(G)} \mathcal C^\bullet(G,S)
\end{equation}
as cochain complexes. Moreover, there is a canonical chain isomorphism mapping each \eqref{eq:7} in $\mathcal C^\bullet(G,S)$ to itself when viewed as an element of $\mathcal C^\bullet(G[S],S)$. Taking cohomology concludes the proof. 
\end{proof}

\begin{rem}\label{rem:21}
The degree of \eqref{eq:7} as an element of $\mathcal C^\bullet(G,S)$ satisfies the bounds
\begin{equation}
    \frac{|S|}{2}\le p+q\le |S|+|E(G[S])|\,.
\end{equation}
The lower bound can be improved to a strict inequality by observing that the terms of $Q(x_{r_1,s_1}^*\cdots x_{r_q,s_q}^*)$ are linearly independent, preventing any possible cancellation. As a result, \eqref{eq:9} can be restated as
\begin{equation}\label{eq:12}
b_i(G)=\sum_{|V(H)|\le 2i-1} \left|\binom{G}{H}\right| \beta_i(H)\,,
\end{equation}
where the sum is over all graphs $H$ with at most $2i-1$ vertices and $\left|\binom{G}{H}\right|$ denotes the number of occurrences of $H$ as an induced subgraph of $G$. In other words, the dimension of the cohomology of $\mathcal L(G)$ can be understood as a weighted count of the induced subgraphs of $G$, with weights given by dimensions of the corresponding essential cohomology groups. 
\end{rem}

\begin{example}
Since the induced subgraphs of complete graphs are necessarily complete, \eqref{eq:12} specializes to
\begin{equation}\label{eq:15}
b_i(K_n)=\sum_{r=1}^n \beta_i(K_r) \binom{n}{r}\,.
\end{equation}
Thus, taking the inverse binomial transform,
\begin{equation}
\beta_i(K_n)=\sum_{r=1}^n (-1)^{n-r} b_i(K_r)\binom{n}{r}\,.
\end{equation}
Since the dimensions of the cohomology groups of the free 2-step nilpotent Lie algebras $\mathcal L(K_n)$ have a natural representation-theoretic interpretation \cites{Kostant61, Sigg96}, the same can be said about the dimensions of the essential cohomology groups of complete graphs.
\end{example}

\begin{rem}
Each subcomplex $\mathcal C^\bullet(G,S)$ further decomposes into a direct sum of subcomplexes $\mathcal C^\bullet_N(G,S)$, labelled by non-negative integers $N$, whose underlying vector space is spanned by all monomials such that, in the notation of \eqref{eq:7}, $N=p+2q$. This gives rise to a bigraded refinement of the essential cohomology $\mathcal H^n(G)=\bigoplus_{r=0}^n \mathcal H^{n,r}(G)$, where
\begin{equation}\label{eq:17}
\mathcal H^{n,r}(G)=H^n(\mathcal C^\bullet_{n+r}(G,V(G)))\,.
\end{equation}
If $\beta_{n,r}(G)$ denotes the dimension of \eqref{eq:17}, then \eqref{eq:12} refines to
\begin{equation}
b_i(G)=\sum_{|V(H)|\le 2i-1} \sum_{r=0}^i\left|\binom{G}{H}\right| \beta_{i,r}(H)\,.
\end{equation}
\end{rem}

\begin{rem}\label{rem:14} According to \eqref{eq:12}, in order to calculate the first and second cohomology of $\mathcal L(G)$, it suffices to calculate the dimensions of the first and second essential cohomology groups of graphs with up to three vertices. It follows from Remark \ref{rem:21} that the only graph with non-zero first essential cohomology is $K_1$ and $\beta_1(K_1)=1$. As for the second cohomology, Remark \ref{rem:19} implies $\beta_2(K_1)=\beta_2(K_3)=\beta_2(K_2+K_1)=0$ and $\beta_2(2K_1)=1$. Moreover, a straightforward calculation shows that $\mathcal H^2(K_2)={\rm span}\{x_1^*x_2^*\}$, $\mathcal H^2(P_3)={\rm span}\{x_1^*x_{2,3}^*+x_{1,2}^*x_3^*\}$, and 
\begin{equation}
    \mathcal H^2(K_3)={\rm span}\{x_i^*x_{i+1,i+2}^*+x_{i,i+1}^*x_{i+2}^*, x_{i+1}^*x_{i+2,i+3}^*+x_{i+1,i+2}^*x_{i+3}^* \}
\end{equation} for any $i\in \mathbb Z/3\mathbb Z$.  As a result, we obtain a new short proof of the following theorem.
\end{rem}

\begin{theorem}[\cite{PouseeleTirao09}]\label{thm:26}
For any finite simple graph $G$, $b_1(G)=|V(G)|$ and
\begin{equation}
b_2(G)=\binom{|V(G)|}{2}+\frac{1}{2}\sum_{i\in V(G)}(\deg(i))^2-\left|\binom{G}{K_3}\right|\,,
\end{equation}
where $\deg(i)$ denotes the degree of the vertex $i$.
\end{theorem}

\begin{proof} The first statement follows immediately from \eqref{eq:12}, since $\binom{G}{K_1}=V(G)$. To prove the second statement, we notice that Remark \ref{rem:14} implies
\begin{equation}
b_2(G)=\left|\binom{G}{2K_1} \right|+2\left|\binom{G}{K_2}\right|+\left|\binom{G}{P_3}\right|+2\left|\binom{G}{K_3}\right|\,.
\end{equation}
Then, noticing that 
\begin{equation}
\left|\binom{G}{2K_1} \right|+2\left|\binom{G}{K_2}\right| = \binom{|V(G)|}{2}+|E(G)|
\end{equation}
and
\begin{equation}
\left|\binom{G}{P_3}\right|+3\left|\binom{G}{K_3}\right|=\sum_{i\in V(G)}\binom{\deg(i)}{2}\,,
\end{equation}
we apply the Handshaking Lemma to conclude the proof.
\end{proof}

\begin{figure}
  \centering
  \begin{subfigure}{0.2\textwidth}
    \centering
    \begin{tikzpicture}
      \foreach \x in {0, 1}
          \fill (\x,0) circle[radius=2pt];
      
      \draw (0,0) -- (1,0);
      
      \node at (0.5, -1) {$ G_1: 1$};
    \end{tikzpicture}
  \end{subfigure}
  \hfill
  \begin{subfigure}{0.2\textwidth}
    \centering
    \begin{tikzpicture}
      \foreach \x in {0, 1}
          \fill (\x,0) circle[radius=2pt];
      \fill (0.5,0.866) circle[radius=2pt];
      
      \node at (0.5, -1) {$G_2: 1$};
    \end{tikzpicture}
  \end{subfigure}
  \hfill
  \begin{subfigure}{0.2\textwidth}
    \centering
    \begin{tikzpicture}
      \foreach \x in {0, 1}
          \fill (\x,0) circle[radius=2pt];
      \fill (0.5,0.866) circle[radius=2pt];
      
      \draw (0,0) -- (1,0);
      
      \node at (0.5, -1) {$G_3: 2$};
    \end{tikzpicture}
  \end{subfigure}
  \hfill
  \begin{subfigure}{0.2\textwidth}
    \centering
    \begin{tikzpicture}
      \foreach \x in {0, 1}
          \fill (\x,0) circle[radius=2pt];
      \fill (0.5,0.866) circle[radius=2pt];
      \draw (0,0) -- (1,0) -- (.5,.866);

      \node at (0.5, -1) {$G_4: 4$};
    \end{tikzpicture}
  \end{subfigure}

\bigskip

  \centering
  \begin{subfigure}{0.2\textwidth}
    \centering
    \begin{tikzpicture}
      \foreach \x in {0, 1}
          \fill (\x,0) circle[radius=2pt];
      \fill (0.5,0.866) circle[radius=2pt];
      
      \draw (0,0) -- (1,0) -- (0.5,0.866) -- cycle;
      
      \node at (0.5, -1) {$G_5: 9$};
    \end{tikzpicture}
  \end{subfigure}
  \hfill
  \begin{subfigure}{0.2\textwidth}
    \centering
    \begin{tikzpicture}
      \foreach \x in {0, 1}
          \fill (\x,0) circle[radius=2pt];
      \foreach \x in {0, 1}
          \fill (\x,1) circle[radius=2pt];
      \draw (0,0) -- (1,0) -- (1,1);
      
      \node at (0.5, -1) {$G_6: 1$};
    \end{tikzpicture}
  \end{subfigure}
  \hfill
  \begin{subfigure}{0.2\textwidth}
    \centering
    \begin{tikzpicture}
      \foreach \x in {0, 1}
          \fill (\x,0) circle[radius=2pt];
      \foreach \x in {0, 1}
          \fill (\x,1) circle[radius=2pt];
      
      \draw (0,0) -- (1,0) -- (1,1) --(0,1);
      
      \node at (0.5, -1) {$G_7: 1$};
    \end{tikzpicture}
  \end{subfigure}
  \hfill
  \begin{subfigure}{0.2\textwidth}
    \centering
    \begin{tikzpicture}
      \foreach \x in {0, 1}
          \fill (\x,0) circle[radius=2pt];
      \foreach \x in {0, 1}
          \fill (\x,1) circle[radius=2pt];
      
      \draw (0,0) -- (1,0) -- (1,1) --(0,1) -- cycle;
      
      \node at (0.5, -1) {$G_8: 5$};
    \end{tikzpicture}
  \end{subfigure}

\bigskip

  \centering
  \begin{subfigure}{0.2\textwidth}
    \centering
    \begin{tikzpicture}
       \foreach \x in {0, 1}
          \fill (\x,0) circle[radius=2pt];
      \foreach \x in {0, 1}
          \fill (\x,1) circle[radius=2pt];
      
      \draw (0,0) -- (1,0) -- (1,1) -- cycle;
      
      \node at (0.5, -1) {$G_9: 2$};
    \end{tikzpicture}
  \end{subfigure}
  \hfill
  \begin{subfigure}{0.2\textwidth}
    \centering
    \begin{tikzpicture}
      \foreach \x in {0, 1}
          \fill (\x,0) circle[radius=2pt];
      \foreach \x in {0, 1}
          \fill (\x,1) circle[radius=2pt];
      \draw (0,0) -- (1,0) -- (1,1);
      \draw (1,0) -- (0,1);
      
      \node at (0.5, -1) {$G_{10}: 2$};
    \end{tikzpicture}
  \end{subfigure}
  \hfill
  \begin{subfigure}{0.2\textwidth}
    \centering
    \begin{tikzpicture}
      \foreach \x in {0, 1}
          \fill (\x,0) circle[radius=2pt];
      \foreach \x in {0, 1}
          \fill (\x,1) circle[radius=2pt];
      
      \draw (0,0) -- (1,0) -- (1,1) --(0,1);
      \draw (0,0) -- (1,1);
      
      \node at (0.5, -1) {$G_{11}: 6$};
    \end{tikzpicture}
  \end{subfigure}
  \hfill
  \begin{subfigure}{0.2\textwidth}
    \centering
    \begin{tikzpicture}
      \foreach \x in {0, 1}
          \fill (\x,0) circle[radius=2pt];
      \foreach \x in {0, 1}
          \fill (\x,1) circle[radius=2pt];
      
      \draw (0,0) -- (1,0) -- (1,1) --(0,1) -- cycle;
      \draw (0,0) -- (1,1);
      
      \node at (0.5, -1) {$G_{12}: 14$};
    \end{tikzpicture}
  \end{subfigure}

\bigskip

  \centering
  \begin{subfigure}{0.2\textwidth}
    \centering
    \begin{tikzpicture}
       \foreach \x in {0, 1}
          \fill (\x,0) circle[radius=2pt];
      \foreach \x in {0, 1}
          \fill (\x,1) circle[radius=2pt];
      
      \draw (0,0) -- (1,0) -- (1,1) -- cycle;
      \draw (1,0) -- (0,1) -- (1,1);
      \draw (0,0) -- (0,1);
      
      \node at (0.5, -1) {$G_{13}: 26$};
    \end{tikzpicture}
  \end{subfigure}
  \hfill
  \begin{subfigure}{0.2\textwidth}
    \centering
    \begin{tikzpicture}
     
      \fill (1.288,0.247+.45) circle[radius=2pt];
      \fill (.991,-.64+.45) circle[radius=2pt];
      \fill (.058,-.64+.45) circle[radius=2pt];
      \fill (-.257,0.247+.45) circle[radius=2pt];
      \fill (.528 ,.8 + .45) circle[radius=2pt];

      \draw (-.257,0.247+.45) -- (.528,.8+.45) -- (1.288,0.247+.45)-- (.991,-.64+.45) -- (0.058,-.64+.45);
      \draw (0.058,-.64+.45) --(.528,.8+.45);
      \draw (-.257,0.247+.45) -- (.991,-.64+.45);

      \node at (0.5, -1) {$G_{14} : 1$};
    \end{tikzpicture}
  \end{subfigure}
  \hfill
  \begin{subfigure}{0.2\textwidth}
    \centering
    \begin{tikzpicture}
      \fill (1.288,0.247+.45) circle[radius=2pt];
      \fill (.991,-.64+.45) circle[radius=2pt];
      \fill (.058,-.64+.45) circle[radius=2pt];
      \fill (-.257,0.247+.45) circle[radius=2pt];
      \fill (.528 ,.8+.45) circle[radius=2pt];

      \draw (-.257,0.247+.45) -- (.528,.8+.45) -- (1.288,0.247+.45)-- (.991,-.64+.45) -- (0.058,-.64+.45);
      \draw (0.058,-.64+.45) --(.528,.8+.45) -- (.991,-.64+.45)--(-.257,0.247+.45);
      
      \node at (0.5, -1) {$G_{15}: 1$};
    \end{tikzpicture}
  \end{subfigure}
  \hfill
  \begin{subfigure}{0.2\textwidth}
    \centering
    \begin{tikzpicture}
      \fill (1.288,0.247+.45) circle[radius=2pt];
      \fill (.991,-.64+.45) circle[radius=2pt];
      \fill (.058,-.64+.45) circle[radius=2pt];
      \fill (-.257,0.247+.45) circle[radius=2pt];
      \fill (.528 ,.8+.45) circle[radius=2pt];

      \draw (-.257,0.247+.45) -- (.528,.8+.45) -- (1.288,0.247+.45)-- (.991,-.64+.45) -- (0.058,-.64+.45) -- cycle;
      
      \node at (0.5, -1) {$G_{16}: 1$};
    \end{tikzpicture}
  \end{subfigure}

\bigskip

  \centering
  \begin{subfigure}{0.2\textwidth}
    \centering
    \begin{tikzpicture}
      
      \fill (1.288,0.247+.45) circle[radius=2pt];
      \fill (.991,-.64+.45) circle[radius=2pt];
      \fill (.058,-.64+.45) circle[radius=2pt];
      \fill (-.257,0.247+.45) circle[radius=2pt];
      \fill (.528 ,.8+.45) circle[radius=2pt];

      \draw (-.257,0.247+.45) -- (.528,.8+.45) -- (1.288,0.247+.45)-- (.991,-.64+.45) -- (0.058,-.64+.45) -- cycle;
      \draw (0.058,-.64+.45) --(.528,.8+.45);

      \node at (0.5, -1) {$G_{17}: 1$};
    \end{tikzpicture}
  \end{subfigure}
  \hfill
  \begin{subfigure}{0.2\textwidth}
    \centering
    \begin{tikzpicture}
     
      \fill (1.288,0.247+.45) circle[radius=2pt];
      \fill (.991,-.64+.45) circle[radius=2pt];
      \fill (.058,-.64+.45) circle[radius=2pt];
      \fill (-.257,0.247+.45) circle[radius=2pt];
      \fill (.528 ,.8+.45) circle[radius=2pt];

      \draw (-.257,0.247+.45) -- (.528,.8+.45) -- (1.288,0.247+.45)-- (.991,-.64+.45) -- (0.058,-.64+.45) -- cycle;
      \draw (0.058,-.64+.45) --(.528,.8+.45);
      \draw (.528,.8+.45) -- (.991,-.64+.45);

      \node at (0.5, -1) {$G_{18} : 1$};
    \end{tikzpicture}
  \end{subfigure}
  \hfill
  \begin{subfigure}{0.2\textwidth}
    \centering
    \begin{tikzpicture}
      \fill (1.288,0.247+.45) circle[radius=2pt];
      \fill (.991,-.64+.45) circle[radius=2pt];
      \fill (.058,-.64+.45) circle[radius=2pt];
      \fill (-.257,0.247+.45) circle[radius=2pt];
      \fill (.528 ,.8+.45) circle[radius=2pt];

      \draw (-.257,0.247+.45) -- (.528,.8+.45) -- (1.288,0.247+.45)-- (.991,-.64+.45) -- (0.058,-.64+.45) -- cycle;
      \draw (0.058,-.64+.45) --(.528,.8+.45);
      \draw (.991,-.64+.45)--(-.257,0.247+.45);
      
      \node at (0.5, -1) {$G_{19}: 2$};
    \end{tikzpicture}
  \end{subfigure}
  \hfill
  \begin{subfigure}{0.2\textwidth}
    \centering
    \begin{tikzpicture}
      \fill (1.288,0.247+.45) circle[radius=2pt];
      \fill (.991,-.64+.45) circle[radius=2pt];
      \fill (.058,-.64+.45) circle[radius=2pt];
      \fill (-.257,0.247+.45) circle[radius=2pt];
      \fill (.528 ,.8+.45) circle[radius=2pt];

      \draw (-.257,0.247+.45) -- (.528,.8+.45) -- (1.288,0.247+.45)-- (.991,-.64+.45) -- (0.058,-.64+.45) -- cycle;
      \draw (0.058,-.64+.45) --(.528,.8+.45);
      \draw (.528,.8+.45) -- (.991,-.64+.45);
      \draw (-.257,0.247+.45) -- (.991,-.64+.45);

      \node at (0.5, -1) {$G_{20}: 2$};
    \end{tikzpicture}
  \end{subfigure}

\bigskip

  \centering
  \begin{subfigure}{0.2\textwidth}
    \centering
    \begin{tikzpicture}
      
      \fill (1.288,0.247+.45) circle[radius=2pt];
      \fill (.991,-.64+.45) circle[radius=2pt];
      \fill (.058,-.64+.45) circle[radius=2pt];
      \fill (-.257,0.247+.45) circle[radius=2pt];
      \fill (.528 ,.8+.45) circle[radius=2pt];

     \draw (-.257,0.247+.45) -- (.528,.8+.45) -- (1.288,0.247+.45)-- (.991,-.64+.45) -- (0.058,-.64+.45) -- cycle;
      \draw (0.058,-.64+.45) --(.528,.8+.45);
      \draw (.528,.8+.45) -- (.991,-.64+.45);
      \draw (-.257,0.247+.45) -- (1.288,0.247+.45);

      \node at (0.5, -1) {$G_{21}: 3$};
    \end{tikzpicture}
  \end{subfigure}
  \hfill
  \begin{subfigure}{0.2\textwidth}
    \centering
    \begin{tikzpicture}
     
      \fill (1.288,0.247+.45) circle[radius=2pt];
      \fill (.991,-.64+.45) circle[radius=2pt];
      \fill (.058,-.64+.45) circle[radius=2pt];
      \fill (-.257,0.247+.45) circle[radius=2pt];
      \fill (.528 ,.8+.45) circle[radius=2pt];

      \draw (-.257,0.247+.45) -- (.528,.8+.45) -- (1.288,0.247+.45)-- (.991,-.64+.45) -- (0.058,-.64+.45) -- cycle;
      \draw (0.058,-.64+.45) --(.528,.8+.45);
      \draw (.528,.8+.45) -- (.991,-.64+.45);
      \draw (-.257,0.247+.45) -- (1.288,0.247+.45);
      \draw (-.257,0.247+.45) -- (.991,-.64+.45);
      
      \node at (0.5, -1) {$G_{22} : 4$};
    \end{tikzpicture}
  \end{subfigure}
  \hfill
  \begin{subfigure}{0.2\textwidth}
    \centering
    \begin{tikzpicture}
      \fill (1.288,0.247+.45) circle[radius=2pt];
      \fill (.991,-.64+.45) circle[radius=2pt];
      \fill (.058,-.64+.45) circle[radius=2pt];
      \fill (-.257,0.247+.45) circle[radius=2pt];
      \fill (.528 ,.8+.45) circle[radius=2pt];

      \draw (-.257,0.247+.45) -- (.528,.8+.45) -- (1.288,0.247+.45)-- (.991,-.64+.45) -- (0.058,-.64+.45) -- cycle;
      \draw (0.058,-.64+.45) --(.528,.8+.45) -- (.991,-.64+.45) -- (-.257,0.247+.45) --(1.288,0.247+.45) -- cycle;

      \node at (0.5, -1) {$G_{23}: 6$};
    \end{tikzpicture}
  \end{subfigure}
  \hfill
  \begin{subfigure}{0.2\textwidth}
    \centering
    \begin{tikzpicture}

    \end{tikzpicture}
  \end{subfigure}
  \caption{The 23 graphs with non-vanishing third essential cohomology and the corresponding values of $\beta_3(G)$.}
  \label{fig:dim}
  
\end{figure}
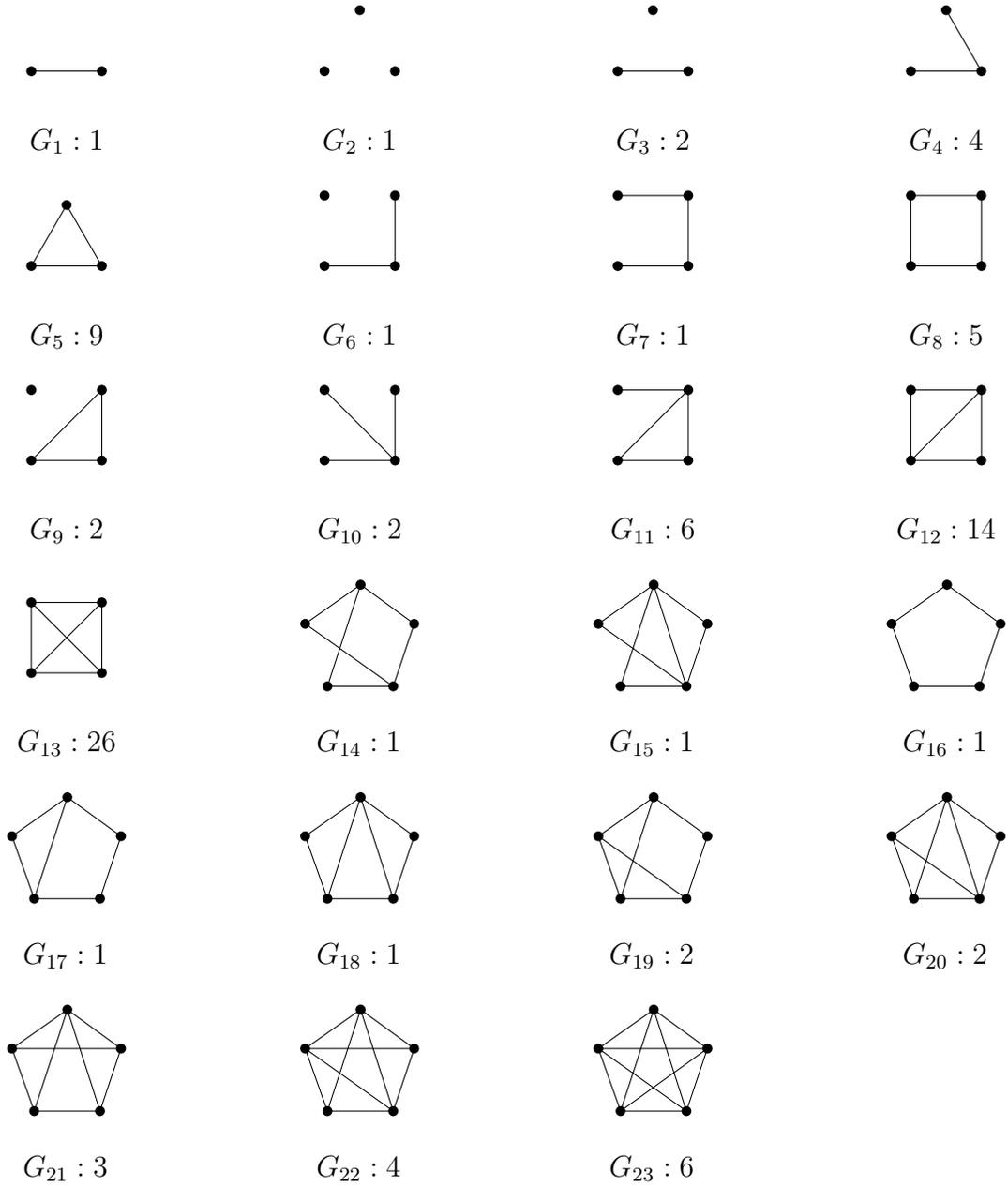

\begin{theorem}\label{thm:25}
Let $G$ be a finite simple graph. Then
\begin{equation}\label{eq:22}
b_3(G)=\sum_{j=1}^{23} \beta_3(G_j)\left|\binom{G}{G_j}\right|
\end{equation}
where $G_1,\ldots,G_{23}$ are the graphs in Figure \ref{fig:dim}. 
\end{theorem}

\begin{proof}
By Remark \ref{rem:21}, only graphs with at most five vertices have non-vanishing cohomology. Moreover $K_1$ and $K_2$ cannot support any monomial of degree 3. The essential cohomology of the remaining 50 (unlabeled) graphs can be calculated by brute force (e.g.\ using SageMath \cite{sagemath}). Of these 50, only the 23 graphs listed in Figure \ref{fig:dim} have a non-vanishing third essential cohomology.
\end{proof}

\begin{example}\label{ex:24}
Let $G$ be the graph with vertices $\{1,\ldots,6\}$ and edges $\{1,2\}$, $\{2,3\}$, $\{3,4\}$, $\{1,4\}$, $\{1,5\}$, $\{5,6\}$ as depicted in Figure \ref{fig:2}. Then $G$ contains the following induced subgraphs: 6 copies of $G_1=P_2$; 3 copies of $G_2=3K_1$; 10 copies of $G_3=K_2+K_1$; 7 copies of $G_4=P_3$; 6 copies of $G_6=P_3+K_1$; 4 copies of $G_7=P_4$; 1 copy of $G_8=C_4$; and 1 copy of $G_{10}=S_3$. Substituting into \eqref{eq:22} and using the values of the third essential cohomology of these induced subgraphs, as tabulated in Figure \ref{fig:dim}, we obtain $b_3(G)=74$.
\end{example}

\begin{example}
Substituting the dimensions of the third essential cohomology groups from Figure \ref{fig:dim} into \eqref{eq:15} we obtain, in accordance with \cites{Kostant61, Sigg96}, 
\begin{equation}
b_3(K_n)=\binom{n}{2}+9\binom{n}{3}+26\binom{n}{4}+6\binom{n}{5} = \frac{3 n^5 + 35 n^4 - 195 n^3 + 325 n^2 - 168n}{60}
\end{equation}
for every $n\ge 1$. 
\end{example}

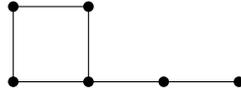
\begin{figure}
    \centering
    \begin{tikzpicture}
     
      \foreach \x in {0,1,2,3}
          \fill (\x,0) circle[radius=2pt];
      \foreach \x in {0,1}
          \fill (\x,1) circle[radius=2pt];

      \draw (0,0) -- (0,1) -- (1,1) -- (1,0) -- cycle;
      \draw (1,0) -- (2,0)--(3,0);

    \end{tikzpicture}
    \caption{The graph $G$ as defined in Example \ref{ex:24}}\label{fig:2}
  \end{figure}

\begin{rem}\label{rem:28}
While surely there is an analogue of Theorem \ref{thm:25} in every degree, in degree 4 it would already involve pre-calculating the fourth essential cohomology of each of the 1249 graphs whose number of vertices is between 3 (the minimum to support a monomial of degree 4) and 7. Hence a general formula quickly becomes impractical in degree four or higher.
\end{rem}

\section{Cohomology of Dani-Mainkar Lie Algebras Associated with Star Graphs}

\begin{rem}
Remark \ref{rem:28} notwithstanding, Theorem \ref{thm:20} can be used to inductively calculate the full cohomology of Dani-Mainkar Lie algebras of special families of graphs. To illustrate this point, we focus on {\it star graphs}
$S_n$ with vertices $V(S_n)=\{1,\ldots,n+1\}$ and edges $E(S_n)=\{\{1,2\},\ldots,\{1,n+1\}\}$. While formulas for the dimensions $b_k(S_n)$ of the cohomology groups are known \cite{ArmstrongCairnsJessup97}, our derivation is different as it involves first calculating the dimensions $\beta_k(S_n)$ of the essential cohomology and leveraging the particularly simple structure of the subgraphs of $S_n$. Our methods can, at least in principle, be applied to more general classes of graphs such as trees while the proof given in \cite{ArmstrongCairnsJessup97} heavily relies on the privileged nature of the vertex 1 in $S_n$.
\end{rem}

\begin{lem}\label{lem:30}
There is a long exact sequence
\begin{equation}\label{eq:dixmier}
\cdots\rightarrow H^{i-1}(\mathcal D^\bullet)\xrightarrow{\delta}H^{i-1}(\mathcal D^\bullet)\xrightarrow{x_{n+1}^*\cdot}H^i(\mathcal L(S_n))\xrightarrow{\pi} H^i(\mathcal D^\bullet)\xrightarrow{\delta} H^{i+1}(\mathcal D^\bullet)\rightarrow \cdots
\end{equation}
where
\begin{enumerate}[i)]
\item $\mathcal D^\bullet$ is the subcomplex of $\mathcal C^\bullet(S_n)$ consisting of polynomials that do not contain the variable $x_{n+1}^*$;
\item $x_{n+1}^*\cdot$ denotes left multiplication by $x_{n+1}^*$;
\item $\pi$ is induced by the natural projection $\frac{\partial}{\partial x_{n+1}^*}\circ x_{n+1}^*\cdot:\mathcal C^\bullet(S_n)\to \mathcal D^\bullet$;
\item $\delta=\frac{\partial}{\partial x_{n+1}^*} \circ Q$.
\end{enumerate}
\end{lem}

\begin{proof}
This is a particular case of the Dixmier exact sequence \cite{Dixmier55}, applied to the Dani-Mainkar construction as in \cite{PouseeleTirao09} (see also \cite{AldiBevins23} for a direct proof adapted to this context). 
\end{proof}

\begin{rem}\label{rem:31}
Since $H^i(\mathcal D^\bullet)$ is canonically isomorphic to the direct sum of $H^i(\mathcal L(S_{n-1}))$ and $x_{1,n+1}^*H^{i-1}(\mathcal L(S_{n-1}))$, then $H^i(\mathcal L(S_n))$ is isomorphic to the direct sum of:
\begin{enumerate}[i)]
\item $x_{n+1}^*{\rm coker}(x_1^*\cdot:H^{i-2}(\mathcal L(S_{n-1}))\to H^{i-1}(\mathcal L(S_{n-1})))$;
\item $x_{n+1}^*x_{1,n+1}^*H^{i-2}(\mathcal L(S_{n-1}))$;
\item $H^i(\mathcal L(S_{n-1}))$;
\item $x_{1,n+1}^*\ker(x_1^*\cdot:H^{i-1}(\mathcal L(S_{n-1}))\to H^i(\mathcal L(S_{n-1})))$.
\end{enumerate}
Direct summands i)-iii) are canonically identified with elements of $H^i(\mathcal L(S_n))$. This is not the case for summand iv) which corresponds to elements of $H^i(\mathcal D^\bullet)$ that lift nontrivially, to $H^i(\mathcal L(S_n))$ under $\pi$.
\end{rem}

\begin{rem}\label{rem:32}
Let us specialize Remark \ref{rem:31} to the case $i=n$ and focus on the essential cohomology. Since monomials in $\mathcal C^\bullet(S_n,V(S_n))$ must be of degree at least $n$ in order to account for all the vertices, the essential cohomology of $S_n$ is concentrated in degrees $\{n,n+1,\ldots,2n+1\}$. It follows that the contribution to $\mathcal H^{n,r}(S_n)$ from i) is isomorphic to $x_{n+1}^*\mathcal H^{n-1,r}(S_{n-1})$, while there is no contribution from ii). Moreover, there is no contribution from iii), since the polynomial representatives $H^n(\mathcal L(S_{n-1}))$ do not depend on $x_{n+1}^*$ or $x_{1,n+1}^*$. Finally, the contribution from iv) depends on $r$. If $r=1$, then $\pi$ maps $\mathcal H^{n,1}(S_n)$ surjectively onto the one-dimensional space  spanned by $x_{1,n+1}^*x_2^*\cdots x_n^*$. Choosing an explicit lift, we can identify this one-dimensional subspace with the span of 
\begin{equation}\label{eq:25}
x_{n+1}^*x_{1,2}^*x_3^*\cdots x_n^*-x_{1,n+1}^*x_2^*x_3^*\cdots x_n^*\,.
\end{equation}
In particular, using induction on $n$, we obtain $\beta_{n,1}(S_n)=n-1$. On the other hand, if $r>1$, then  $\pi\circ x_{1,n+1}^*\cdot$ maps $\mathcal H^{n,r}(S_n)$ onto the kernel of the multiplication map $x_1^*\cdot$ restricted to $\mathcal H^{n-1,r-1}(S_{n-1})$. Hence, for $1<r\le n$,
\begin{equation}\label{eq:26}
\beta_{n,r}(S_n)=\beta_{n-1,r}(S_{n-1})+\beta_{n-1,r-1}(S_{n-1})-\rho_{n,r}\,,
\end{equation}
where $\rho_{n,r}$ is the rank of $x_1^*\cdot:\mathcal H^{n-1,r-1}(S_{n-1})\to \mathcal H^{n,r-1}(S_{n-1})$.
\end{rem}

\begin{lem}\label{lem:33} 
For every $n>0$ and $0\le r\le n/2$,  $\mathcal H^{n,r}(S_n)$ is spanned by polynomials of the form
\begin{equation}\label{eq:27}
x_{i_1}^*\cdots x_{i_p}^*\eta_{s_1,t_1}\cdots \eta_{s_r,t_r}\,,
\end{equation}
where $(i_1,\ldots,i_p,s_1,t_1,\ldots,s_q,t_q)$ is a permutation of $\{2,\ldots,n+1\}$ and $\eta_{a,b}=x_{a}^*x_{1,b}^*+x_b^*x_{1,a}^*$ for all $a,b\in\{1,\ldots,n\}$.  
\end{lem}

\begin{proof}
We prove this by induction on $n$, with the case $n=1$ being trivially true since $\beta_1(S_1)=0$. Assuming the results hold for $\mathcal H^{n-1,r-1}(S_{n-1})$ and $\mathcal H^{n-1,r}(S_{n-1})$, we can follow the analysis carried out in Remark \ref{rem:32}. Since \eqref{eq:25} can be written as $x_3^*\cdots x_n^*\eta_{2,n+1}$, it is already in the form \eqref{eq:27}. Moreover, by induction, $x_{n+1}^*\mathcal H^{n-1,r}(S_{n-1})$ is spanned by polynomials of the form \eqref{eq:27}. Furthermore, every element of $\mathcal H^{n-1,r-1}(S_{n-1})$ is a sum of terms of the form $x_i^*\omega$ for some $i\in \{2,\ldots,n\}$. This implies that $\omega$ is a cocycle in the complex $\mathcal C^\bullet_{n+r-3}(S_{n-1},V(S_{n-1})\setminus\{i\})$, and thus can (up to relabeling of the vertices) be thought of as an element of $\mathcal H^{n-2,r-1}(S_{n-2})$. Then $x_1^*x_i^*\omega=Q(x_{1,i}^*\omega)$ vanishes in cohomology, and thus $x_i^*\omega$ contributes to $\mathcal H^{n,r}(S_n)$. Namely, $x_{1,n+1}^*x_i^*\omega=\pi(-\eta_{i,n+1}\omega)$ and, by induction, $\eta_{i,n+1}\omega$ is (up to a possible sign) of the form \eqref{eq:27}. To conclude the proof, we are going to show that when $n$ is odd, the multiplication operator $x_1^*\cdot$ has trivial kernel when restricted to $\mathcal H^{n-1,(n-1)/2}(S_{n-1})$. By induction, every element of $\mathcal H^{n-1,(n-1)/2}(S_{n-1})$ can be written as
\begin{equation}\label{eq:28}
P=\eta_{2,3}P_{2,3}+{\sum_{i,j}}'  \eta_{2,i}\eta_{3,j}P_{i,j}\,,
\end{equation}
where the sum is over all {\it distinct} $i,j\in \{3,\ldots,n-1\}$ and each $P_{i,j}$ is a polynomial of degree $n-3$ in the variables $\eta_{r,s}$ with $r,s\notin \{i,j\}$. Suppose $x_1^*P$ vanishes in $\mathcal H^{n,(n-1)/2}(S_{n-1})$ i.e.\ 
\begin{equation}\label{eq:29}
x_1^*P=Q(x_2^*x_3^*\Gamma_{23}+x_2^*x_{1,3}^*\Gamma_2+x_3^*x_{1,2}^*\Gamma_3+x_{1,2}^*x_{1,3}^*\Gamma)
\end{equation}
for some $\Gamma_{23},\Gamma_2,\Gamma_3,\Gamma\in \mathcal C^\bullet(S_{n-3},V(S_{n-3})\setminus\{2,3\})$. Expanding the RHS of \eqref{eq:29} using the Leibniz Rule and matching coefficients with the LHS, we obtain
\begin{equation}\label{eq:30}
x_1^*{\sum_{i,j}}' x_j^*x_{1,i}^* P_{i,j}=-Q(\Gamma_2)-x_1^*\Gamma 
\end{equation}
and 
\begin{equation}\label{eq:31}
x_1^*{\sum_{i,j}}' x_i^*x_{1,j}^* P_{i,j}=-Q(\Gamma_3)+x_1^*\Gamma\,. 
\end{equation}
Adding \eqref{eq:30} and \eqref{eq:31} yields
\begin{equation}
x_1^*{\sum_{i,j}}'\eta_{i,j}P_{i,j}=Q(-\Gamma_2-\Gamma_3)\,.
\end{equation}
Using the induction hypotheses and keeping into account that all $\eta_{i,j}P_{i,j}$ are linearly independent, we conclude that each $P_{i,j}$ vanishes. Substituting into \eqref{eq:28}, \eqref{eq:29} and arguing as above, we obtain $x_1^*\eta_{2,3}P_{2,3}=Q(\Gamma_2+\Gamma_3)$ which contradicts the induction hypothesis. Hence, an element of $\mathcal H^{n-1,r}(S_{n-1})$ is in the kernel of $x_1^*\cdot$ if and only if $2r<n-1$. Moreover, all such elements are projections of linear combinations of polynomials of the form \eqref{eq:27}. 
\end{proof}

\begin{lem}
For every $n> 0$, $\beta_n(S_n)=\binom{n}{\lfloor n/2\rfloor}-1$.
\end{lem}

\begin{proof}
Let $\rho_{n,r}$ be as in \eqref{eq:26}. The proof of Lemma \ref{lem:33} shows that $\rho_{n,r}=0$, unless $n$ is odd and $r=\frac{n+1}{2}$, in which case $\rho_{n,\frac{n+1}{2}}=\beta_{n-1,\frac{n-1}{2}}(S_{n-1})$. Substituting into \eqref{eq:26}, we obtain the recursion relation
\begin{equation}\label{eq:35}
    \beta_{n,r}(S_n)=\beta_{n-1,r}(S_{n-1})+\beta_{n-1,r-1}(S_{n-1})
\end{equation}
for all $\frac{n}{2}>r>1$, subject to $\beta_{n,r}(S_n)=0$ if $r>\frac{n}{2}$ and $\beta_{n,1}(S_n)=n-1$ for all $n$. As it is easy to verify, the only solution to this recursion equation is given by $\beta_{n,r}(S_n)=\binom{n}{r}-\binom{n}{r-1}$ whenever $\frac{n}{2}>r\ge 1$, and $\beta_{n,r}(S_n)=0$ otherwise. Hence, 
\begin{equation}
\beta_n(S_n)=\sum_{r=1}^{\lfloor n/2\rfloor} \beta_{n,r}(S_n)=\binom{n}{\lfloor n/2\rfloor}-1\,.
\end{equation}
\end{proof}

\begin{rem}
Using the Dixmier exact sequence \eqref{eq:dixmier} as in the proof of Lemma \ref{lem:33} shows that $\mathcal H^k(S_n)$ is spanned by polynomials of two types. Type I polynomials are of the form 
\begin{equation}\label{eq:38}
x_{i_1}^*\cdots x_{i_p}^*\eta_{j_1,j_1}\cdots \eta_{j_v,j_v}\eta_{s_1,t_1}\cdots \eta_{s_r,t_r}\,,
\end{equation}
where $p+2v+2r=k$, $(i_1,\ldots,i_p,j_1,\ldots,j_v,s_1,t_1,\ldots,s_q,t_q)$ is a permutation of $\{2,\ldots,n+1\}$, and $\eta_{a,b}=x_a^*x_{1,b}^*+x_b^*x_{1,a}^*$.  Type II polynomials are of the form
\begin{equation}\label{eq:39}
x_{i_1}^*\cdots x_{i_p}^*x_1^*x_{1,j_1}^*\cdots x_{1,j_q}^*\,,
\end{equation}
where $\{i_1,\ldots,i_p\}\cup\{j_1,\ldots,j_q\}=\{2,\ldots,n+1\}$ and $|\{i_1,\ldots,i_p\}\cap\{j_1,\ldots,j_q\}|=k-n$. Since the image of $Q$ is itself generated by monomials divisible by $x_1^*$, then $\mathcal H^k(S_n)$ decomposes as the direct sum of $\mathcal H_I^k(S_n)$ (the subspace generated by type I monomials) and $\mathcal H_{II}^k(S_n)$ (the subspace generated by type II monomials). We denote by $\beta_k^{I}(S_n)$ (resp.\ $\beta_k^{II}(S_n)$) the dimension of $\mathcal H_I^{k}(S_n)$ (resp.\ of $\mathcal H_{II}^{k}(S_n)$). 
\end{rem}

\begin{lem}\label{lem:37}
For every $n\ge 1$, 
$\beta_{n+1}^{II}(S_n)=\binom{n}{\lfloor n/2\rfloor}$.
\end{lem}

\begin{proof} Let $\mathcal H_{II}^{k,r}(S_n)=\mathcal H_{II}^k(S_n)\cap \mathcal H^{k,r}(S_n)$ and let $\beta_{k,r}^{II}(S_n)$ be its dimension. 
Since $\mathcal H_{II}^{n+1,n}(S_n)$ is spanned by $x_1^*x_{1,2}^*\cdots x_{1,n}^*$, then $\beta_{n+1,n}^{II}(S_n)=1$ for all $n\ge 0$. Following the analysis of Remark \ref{rem:32}, we analyze the contributions to $\mathcal H^{n+1,r}(S_n)$ coming from the Dixmier exact sequence. Clearly, there is no contribution from $H^n(\mathcal L(S_{n-1}))$. Moreover, we know from Lemma \ref{lem:33} that there is no contribution from $x_{n+1}^*x_{1,n+1}^*\mathcal H^{n-1}(S_{n-1})$, because none of its elements are of type II. Finally, $\pi$ sends type II polynomials to type II polynomials which are thus in the kernel of $x_1^*\cdot$. Hence $\mathcal H^{n+1,r}_{II}(S_n)$ is isomorphic to  the direct sum of $x_{1,n+1}^*\mathcal H_{II}^{n,r-1}(S_{n-1})$ and the quotient 
\begin{equation}
\frac{x_{n+1}^*\mathcal H_{II}^{n,r}(S_{n-1})}{x_{n+1}^*x_1^*\mathcal H^{n-1,r}(S_{n-1})}\,.
\end{equation}
Therefore,
\begin{equation}\label{eq:43}
\beta_{n+1,r}^{II}(S_n)=\beta_{n,r-1}^{II}(S_{n-1})+\beta_{n,r}^{II}(S_{n-1})-\rho_{n,r+1} 
\end{equation}
for all $n\ge r\ge 1$ subject to $\beta_{n+1,n}^{II}(S_n)=1$ for all $n\ge 1$. Since 
\begin{equation}
\rho_{n,r}=\delta_{r,\frac{n+1}{2}}\beta_{n-1,\frac{n-1}{2}}(S_{n-1})=\delta_{r,\frac{n+1}{2}}\beta_{n,\frac{n-1}{2}}^{II}(S_{n-1})\,,
\end{equation}
the recursion \eqref{eq:43} is solved by 
\begin{equation}
\beta_{n+1,r}^{II}(S_n) = 
\begin{cases}
\binom{n}{n-r}-\binom{n}{n-r-1} & \textrm{ if } \frac{n-1}{2} < r \le n \\
0 & \textrm{ otherwise } \,.
\end{cases}
\end{equation}
Thus,
\begin{equation}
\beta_{n+1}^{II}(S_n)=\sum_{r=\lfloor n/2\rfloor}^n \beta_{n+1,r}^{II}(S_n) = \binom{n}{\lfloor n/2\rfloor}\,.
\end{equation}

\end{proof}

\begin{prop}
For every $k> n>0$,
\begin{equation}\label{eq:39bis}
\beta_k(S_n) = \binom{2n-k}{\lfloor(2n-k)/2\rfloor}\binom{n}{2n-k}+\binom{2n-k+1}{\lfloor(2n-k+1)/2\rfloor}\binom{n}{2n-k+1}\,.
\end{equation}
\end{prop}

\begin{proof}
 By Lemma \ref{lem:33}, every polynomial of type I is obtained from either $x_{i_1}^*\cdots x_{i_p}^*$ or
\begin{equation}
H^{2n-k}(\mathcal C^\bullet(S_n,V(S_n)\setminus\{j_1,\ldots,j_v\}))\cong \mathcal H^{2n-k}(S_{2n-k})
\end{equation}
upon repeated multiplication by elements of the form $\eta_{a,a}$, i.e.\ from repeated use of ii) as in Remark \ref{rem:31}. Hence, the corresponding polynomials are all linearly independent in $\mathcal H^k(S_n)$. Therefore,
\begin{equation}\label{eq:45}
\beta_k^I(S_n)=(\beta_n(S_n)+1)\binom{n}{k-n}=\binom{2n-k}{\lfloor (2n-k)/2\rfloor}\binom{n}{2n-k}\,. 
\end{equation}
Similarly, $\mathcal H_{II}^k(S_n)$ is generated by polynomials of the form $\eta_{a_1,a_1}\cdots \eta_{a_{k-n-1},a_{k-n-1}}\omega$ for some type II polynomial $\omega$ in
\begin{equation}
H^{2n-k+2}(\mathcal C^\bullet(S_n,V(S_n)\setminus \{a_1,\ldots,a_{k-n-1}\}))\cong \mathcal H^{2n-k+2}(S_{2n-k+1})\,.
\end{equation}
Hence, using Lemma \ref{lem:37}, we obtain
 \begin{equation}\label{eq:47}
\beta_{k}^{II}(S_n)=\beta_{2n-k+2}^{II}(S_{2n-k+1})\binom{n}{2n-k+1}=\binom{2n-k+1}{\lfloor (2n-k+1)/2\rfloor}\binom{n}{2n-k+1}\,.
 \end{equation}
Since $\beta_k(S_n)=\beta_k^I(S_n)+\beta_k^{II}(S_n)$, \eqref{eq:39bis} follows from \eqref{eq:45} and \eqref{eq:47}. 
\end{proof}

\begin{theorem}[\cite{ArmstrongCairnsJessup97}]
For every $n,k\ge 0$, $b_k(S_n)=\binom{n+1}{\lceil k/2\rceil}\binom{n}{\lfloor k/2\rfloor}$.
\end{theorem}

\begin{proof}
The induced subgraphs of $S_n$ are of the form $iK_1$ for $i\in\{1,\ldots,n\}$ or $S_j$ for some $j\in\{0,\ldots,n\}$.  Hence, it follows from \eqref{eq:12} that
\begin{align*}
b_k(S_n)&=\sum_{i=1}^n \binom{n}{i} \beta_k(iK_1)+\sum_{j=1}^n \binom{n}{j}\beta_k(S_j)\\
&= \binom{n}{k}+\binom{n}{k}\left(\binom{k}{\lfloor k/2\rfloor} - 1 \right) + \\
&+\sum_{j=k+1}^n \binom{n}{j} \left(\binom{j}{2j-k}\binom{2j-k}{\lfloor (2j-k)/2\rfloor} +\binom{j}{2j-k+1}\binom{2j-k+1}{\lfloor (2j-k+1)/2 \rfloor} \right)\\
& = \sum_{j=0}^n \binom{n}{j,2j-k,j-\lceil k/2\rceil,j-\lfloor k/2\rfloor} \\
&+\sum_{j=0}^n\binom{n}{j,2j-k+1,j-\lfloor k/2\rfloor, j-\lceil k/2\rceil +1}\,.
\end{align*}
The result then follows using the Multinomial Theorem on both sides of
\begin{equation}
(1+x)^{n+1}(1+y)^n=(1+x)(1+x+y+xy)^n
\end{equation}
and equating the coefficients of $x^{\lceil k/2\rceil}y^{\lfloor k/2\rfloor}$.

\end{proof}

\begin{cor}
Let $t(S_n)$ be the dimension of $H^\bullet(\mathcal L(S_n))$ as a real vector space (i.e.\ forgetting the grading). Then $t(S_n)=2\binom{2n+1}{n}$ for every $n\ge 1$.
\end{cor}

\begin{proof}
By Poincar\'e duality, whose validity is guaranteed by \cite{Nomizu54},
\begin{equation}
t(S_n)=2\sum_{l=0}^n b_{2l}(S_n) = 2\sum_{l=0}^n\binom{n+1}{l}\binom{n}{l}\,.
\end{equation}
The result then follows by using the Binomial Theorem on both sides of $(1+x)^{2n+1}=(1+x)^{n+1}(1+x)^n$ and equating the coefficients of $x^n$.
\end{proof}

\begin{rem}
Using Stirling's approximation, we obtain $t(S_n)\sim \frac{4^{n+1}}{\sqrt{\pi n}}$. This growth is intermediate between $t(nK_1)= 2^n$ and $t(K_n)=2^{n^2/2}n^{1/8}\kappa$ for a constant $\kappa$ \cite{GrassbergerKingTirao02}.  
\end{rem}

\section{The Cohomology of the Grantcharov-Grantcharov-Iliev Lie Algebras}

\begin{rem}
Let $G$ be a finite simple graph, and let $\Sigma=\{\sigma_1,\ldots,\sigma_s\}$ be a collection of cliques in $G$.
We denote by $\mathcal C^\bullet(G,\Sigma)$ the Cartan-Chevalley-Eilenberg complex of the Lie algebra $\mathcal L(G,\Sigma)$. Its underlying vector space can be thought of as skew-commuting real polynomials in the variables $x    _i^*$ for all $i\in V(G)$, $x_{i,j}^*$ for all $\{i,j\}\in E(G)$, and $y_k^*$ for all $k\in\{1,\ldots,s\}$. The differential is the odd derivation $Q$ such that
\begin{align*}
Q(x_i^*)&=\sum_{k=1}^s |\{i\}\cap V(\sigma_k)| x_i^*y_k^*\,,\\
Q(x_{i,j}^*)&=x_i^*x_j^*+\sum_{k=1}^s |\{i,j\}\cap V(\sigma_k)| x_{i,j}^* y_k^*\,, 
\end{align*}
and $Q(y_k^*)=0$ for all $k\in\{1,\ldots,s\}$. 
\end{rem}

\begin{lem}\label{lem:43}
    Let $G$ be a finite simple graph and let $\Sigma=\{\sigma_1,\ldots,\sigma_s\}$ be a collection of cliques in $G$ ordered in such a way that $\sigma_i\subseteq \sigma_m$ if and only if $i\le m$.  Let $G'$ be the subgraph of $G$ induced by the vertices that are not contained in $\sigma_m$ and let $\Sigma'=\{\sigma_{m+1}',\ldots,\sigma_s'\}$ be the collection of cliques of $G'$ such that $V(\sigma'_j)=V(\sigma_j)\cap V(G')$ for all $j\in \{m+1,\ldots,s\}$. Then for all $k\ge 0$,
    \begin{equation}
    H^k(\mathcal L(G,\Sigma))\cong H^k(\mathcal L(G'+mK_1,\Sigma'))\,.  
    \end{equation}
\end{lem}

\begin{proof}
Consider the exact sequence of cochain complexes
\begin{equation}\label{eq:47bis}
0\to y_m^*\mathcal C^\bullet(G,\Sigma)\to \mathcal C^\bullet(G,\Sigma)\to \frac{\mathcal C^\bullet(G,\Sigma)}{y_m^*\mathcal C^\bullet(G,\Sigma)}\to 0\,.
\end{equation}
Let $\delta_k$ be the graded components of the connecting homomorphism $\delta$ in the long exact sequence associated with \eqref{eq:47bis}. Hence $H^k(\mathcal L(G,\Sigma))$ is isomorphic to the direct sum of the kernel of $\delta_k$ and the cokernel of $\delta_{k-1}$. A straightforward diagram chase shows that $\delta$ is the odd derivation such that $\delta(x_i^*)=|\{i\}\cap V(\sigma_1)|x_i^*y_1^*$ and $\delta(x_{i,j}^*)=|\{i,j\}\cap V(\sigma_1)|x_{i,j}^*y_m^*$. Using the canonical decomposition of Section \ref{sec:3}, we obtain
\begin{equation}
\mathcal C^\bullet(G,\Sigma) \cong \bigoplus_{S\subseteq V(G)} \mathcal C^\bullet(G,\Sigma,S)
\end{equation}
labeled by the induced subgraphs $S$ of $G$,
where $\mathcal C^\bullet(G,\Sigma,S)$ is the subcomplex with underlying vector space $\mathcal C^\bullet(G,S)\otimes \mathbb R[y_1^*,\ldots,y_s^*]$. Let $P=x_{i_1}^*x_{i_2}^*\cdots x_{i_p}^*x_{r_1,t_1}^*,\ldots,x_{r_q,t_q}^*$ be a monomial in $\mathcal C^\bullet(G,S)$. For every $j\in V(G)$, let $\epsilon_j(P)$ be the number of times $j$ occurs in the sequence $(i_1,\ldots,i_p,r_1,t_1,\ldots,r_q,t_q)$, and let $\epsilon(P)=\sum_{j\in V(\sigma_1)}\epsilon_j(P)$. Then $\delta(P)=-y_m^*\epsilon(P)P$. Hence, 
\begin{align}
\ker(\delta_k)&\cong\bigoplus_{S\cap V(\sigma_1)=\emptyset} H^k(\mathcal C^\bullet(G,S)\otimes \mathbb R[y_1^*,\ldots,\widehat{y_m^*},\ldots y_s^*])\\
&\cong H^k(\mathcal L(G',\Sigma')) \otimes \mathbb R[y_1^*\,\ldots,y_{m-1}^*]\,.
\end{align}
Similarly,
\begin{align}
{\rm coker}(\delta_{k-1})& \cong\bigoplus_{S\cap V(\sigma_1)=\emptyset} y_m^* H^{k-1}(\mathcal C^\bullet(G,S)\otimes \mathbb R[y_1^*,\ldots,\widehat{y_m^*}\,\ldots,y_s^*])\\
&\cong y_m^*H^{k-1}(\mathcal L(G',\Sigma'))\otimes \mathbb R[y_1^*\,\ldots,y_{m-1}^*]\,.
\end{align}
The result follows by applying Remark \ref{rem:10} to $\mathcal L(G'+mK_1,\Sigma')$ and then using by K\"unneth's formula.
\end{proof}

\begin{theorem}\label{isothm}
    Let $G$ be a finite simple graph, let $\Sigma$ be a collection of cliques of $G$, and let $\widetilde G$ be the graph induced by the vertices that do not belong to any clique of $\Sigma$. Then for all $k\geq 0$,
    \begin{equation}
      H^k(\mathcal L(G,\Sigma))\cong H^k(\mathcal L(\widetilde G+|\Sigma|K_1))\,.
    \end{equation}
\end{theorem}
\begin{proof}
   Iterate Lemma \ref{lem:43} until there are no cliques left.  
\end{proof}

 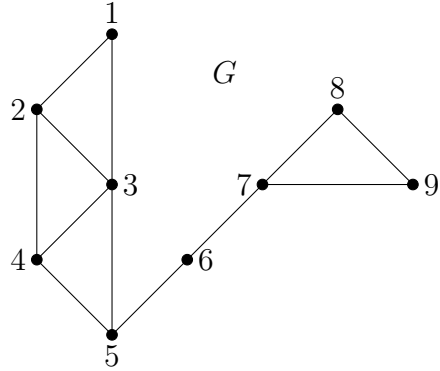
\begin{figure}\begin{tikzpicture}
    \filldraw
    (1,1) circle (2pt) node[above] {$1$}
    (0,0) circle (2pt) node[left] {$2$}
    (1,-1) circle (2pt) node[right] {$3$}
    (0,-2) circle (2pt) node[left] {$4$}
    (1,-3) circle (2pt) node[below] {$5$}
    (2,-2) circle (2pt) node[right] {$6$}
    (3,-1) circle (2pt) node[left] {$7$}
    (4,0) circle (2pt) node[above] {$8$}
    (5,-1) circle (2pt) node[right] {$9$};

    \node (graph) at (2.5,.5) {$G$};
    
    \draw (1,1) -- (0,0) -- (1,-1) -- (1,1);
    \draw (0,0) -- (0,-2) -- (1,-1);
    \draw (0,-2) -- (1,-3) -- (1,-1);
    \draw (1,-3) -- (2,-2) -- (3,-1) -- (4,0) -- (5,-1) -- (3,-1);

    \end{tikzpicture}
    \caption{The graph $G$ as defined in Example \ref{ex:42}.}\label{fig:5}
    \end{figure}

\begin{example}\label{ex:42}
    Let $G$ be as in Figure \ref{fig:5} and let $\Sigma=\{\sigma_1,\sigma_2,\sigma_3,\sigma_4\}$, where $\sigma_1=\{1,2,3\}$, $\sigma_2=\{2,3,4\}$, $\sigma_3=\{3,4,5\}$, $\sigma_4=\{7,8,9\}$. Then $\widetilde G$ is the graph consisting of the single vertex $6$ and thus $H^n(\mathcal L(G,\Sigma))\cong H^n(\mathcal L(5K_1))\cong \mathbb R^{\binom{5}{n}}$ for every $n\ge 0$.    
   
\end{example}

\begin{rem} Let $b_n(G,\Sigma)$ denote the dimension of $H^n(\mathcal L(G,\Sigma))$. As a consequence of Theorem \ref{isothm} and K\"unneth's formula we have
\begin{equation}
b_n(G,\Sigma)=\sum_{l=0}^n b_l(\widetilde G) \binom{|\Sigma|}{l}\,.
\end{equation}
In particular, using Theorem \ref{thm:26}, $b_1(G,\Sigma)=|V(\widetilde G)|+|\Sigma|$, where $V(\widetilde G)$ is precisely the set of vertices of $G$ that do not belong to any clique of $\Sigma$. Similarly,
\begin{equation}
b_2(G,\Sigma)=\binom{|V(\widetilde G)|}{2}+\frac{1}{2}\sum_{i\in V(\widetilde G)}(\widetilde \deg(i))^2-\left|\binom{\widetilde G}{K_3}\right| + |V(\widetilde G)||\Sigma|+\binom{|\Sigma|}{2}\,,
\end{equation}
where $\widetilde \deg(i)$ denotes the number of edges of $G$ containing $i$ that do not share any vertex with the cliques of $\Sigma$ and $\binom{\widetilde {G}}{K_3}$ denotes the set of triangles in $G$ that do not share any vertex with the cliques of $\Sigma$. Along the same lines, by combining Theorem \ref{thm:25} with Theorem \ref{isothm}, we obtain an explicit formula for the third cohomology of any Grantcharov-Grantcharov-Iliev Lie algebra.
\end{rem}

\begin{rem}
Let $\mathcal S(G,\Sigma)$ be a compact solvmanifold  $\exp(\mathcal L(G,\Sigma))/\Lambda$ for a suitable discrete co-compact group $\Lambda$. In particular, $\mathcal S(G,\emptyset)$ is always a compact nilmanifold. Since by Proposition \ref{prop:10} $\mathcal L(G,\Sigma)$ is completely solvable, then the de Rham cohomology of $\mathcal S(G,\Sigma)$ is isomorphic to $H^\bullet(\mathcal L(G,\Sigma))$ by a theorem of Hattori \cite{Hattori60}. Therefore, Theorem \ref{isothm} can be interpreted topologically as stating that for each compact solvmanifold of the form $\mathcal S(G,\Sigma)$ there exists a compact nilmanifold $\mathcal S(\widetilde G+|\Sigma|K_1,\emptyset)$, such that
\begin{equation}
H^\bullet_{\rm dR}(\mathcal S(G,\Sigma))\cong H^\bullet_{\rm dR} (S(\widetilde G+|\Sigma|K_1,\emptyset))
\end{equation}
as graded vector spaces.
\end{rem}

\vskip.1in\noindent
\address{Marco Aldi\\
Department of Mathematics and Applied Mathematics\\
Virginia Commonwealth University\\
Richmond, VA 23284, USA\\
\email{maldi2@vcu.edu}}

\vskip.1in\noindent
\address{Andrew Butler\\
Department of Mathematics \\
University of Wisconsin - Madison\\
Madison, WI 53706, USA\\
\email{ajbutler404@gmail.com}}

\vskip.1in\noindent
\address{Jordan Gardiner\\
Department of Mathematics \\
Brigham Young University\\
Provo, UT 83460, USA\\
\email{jordancorbygardiner@gmail.com}}

\vskip.1in\noindent
\address{Daniele Grandini\\
Department of Mathematics and Economics\\
Virginia State University\\
Petersburg, VA 23806, USA\\
\email{dgrandini@vsu.edu}}

\vskip.1in\noindent
\address{Monica Lichtenwalner\\
Department of Mathematics \\
Virginia Polytechnic Institute and State University\\
Blacksburg, VA 24061, USA\\
\email{mlichtenwalner@vt.edu}}

\vskip.1in\noindent
\address{Kevin Pan\\
Department of Mathematics\\
New York University\\
New York, NY 10012, USA\\
\email{kp2832@nyu.edu}}

\end{document}